\documentclass[]{interact}
\usepackage{amsmath}
\usepackage{epstopdf}
\usepackage{subfigure}
\DeclareMathOperator{\real}{Re}
\usepackage[numbers,sort&compress]{natbib}
\bibpunct[, ]{[}{]}{,}{n}{,}{,}
\renewcommand\bibfont{\fontsize{10}{12}\selectfont}
\makeatletter
\def\NAT@def@citea{\def\@citea{\NAT@separator}}
\makeatother

\theoremstyle{plain}
\newtheorem{theorem}{Theorem}[section]
\newtheorem{lemma}[theorem]{Lemma}
\newtheorem{corollary}[theorem]{Corollary}
\newtheorem{proposition}[theorem]{Proposition}

\theoremstyle{definition}

\theoremstyle{remark}
\newtheorem{remark}{Remark}

\begin{document}
\title{Starlikeness of the generalized  Bessel function}

\author{
\name{Rosihan M. Ali\textsuperscript{a}, See Keong Lee\textsuperscript{b} and Saiful R. Mondal\textsuperscript{c} \thanks{CONTACT See Keong Lee. Email: sklee@usm.my}}
\affil{\textsuperscript{a, b}School of  Mathematical Sciences,
Universiti Sains Malaysia,  11800 USM Penang, Malaysia; \textsuperscript{c}Department of Mathematics and Statistics, Collage of Science,
King Faisal University, Al-Hasa 31982, Hofuf, Saudi Arabia.}
}

\maketitle

\begin{abstract}
For a fixed $a \in \{1, 2, 3, \ldots\},$ the radius of starlikeness of positive order is obtained for each of the normalized analytic functions
\begin{align*}
\mathtt{f}_{a, \nu}(z)&:= \bigg(2^{a \nu-a+1} a^{-\frac{a(a\nu-a+1)}{2}} \Gamma(a \nu+1) {}_a\mathtt{B}_{2a-1, a \nu-a+1, 1}(a^{a/2} z)\bigg)^{\tfrac{1}{a \nu-a+1}},\\
\mathtt{g}_{a, \nu}(z)&:= 2^{a \nu-a+1} a^{-\frac{a}{2}(a\nu-a+1)} \Gamma(a \nu+1) z^{a-a\nu} {}_a\mathtt{B}_{2a-1, a \nu-a+1, 1}(a^{a/2} z),\\
\mathtt{h}_{a, \nu}(z)&:= 2^{a \nu-a+1} a^{-\frac{a}{2}(a\nu-a+1)} \Gamma(a \nu+1) z^{\frac{1}{2}(1+a-a\nu)} {}_a\mathtt{B}_{2a-1, a \nu-a+1, 1}(a^{a/2} \sqrt{z})
\end{align*}
in the unit disk, where ${}_a\mathtt{B}_{b, p, c}$ is the generalized Bessel function
\begin{align*}
{}_a\mathtt{B}_{b, p, c}(z):= \sum_{k=0}^\infty \frac{(-c)^k}{k! \; \mathrm{\Gamma}{\left( a k +p+\frac{b+1}{2}\right)} } \left(\frac{z}{2}\right)^{2k+p}.
\end{align*}
The best range on $\nu$ is also obtained for a fixed $a$ to ensure the functions $\mathtt{f}_{a, \nu}$ and $\mathtt{g}_{a, \nu}$ are starlike of positive order in the unit disk. When $a=1,$ the results obtained reduced to earlier known results.
\end{abstract}
\begin{amscode}
  33C10; 30C45
 \end{amscode}

\maketitle


\section{Introduction}

There is a vast literature describing the importance and applications of the Bessel function of the first kind of order $p$ given by
\begin{equation*}\label{Bessel}
\mathtt{J}_{p}(x):=  \sum_{k=0}^\infty \frac{(-1)^k}{k! \; \mathrm{\Gamma}{\left(k+p+1\right)} } \left(\frac{x}{2}\right)^{2k+p},
\end{equation*}
where $\mathrm{\Gamma}$ is the familiar gamma function. Various generalizations of the Bessel function have also been studied. Perhaps a more complete generalization is that given by Baricz in \cite{Baricz-Galues-Bessel}. In this case, the generalized Bessel function takes the form
\begin{align}\label{eqn:bessel-definition}
{}_a\mathtt{B}_{b, p, c}(x):= \sum_{k=0}^\infty \frac{(-c)^k}{k! \; \mathrm{\Gamma}{\left( a k +p+\frac{b+1}{2}\right)} } \left(\frac{x}{2}\right)^{2k+p}
\end{align}
for $a \in\mathbb{N}=\{1, 2, 3, \ldots\},$ and $b, p, c, x \in \mathbb{R}$. It is evident that the function ${}_a\mathtt{B}_{b, p, c}$ converges absolutely at each $x \in \mathbb{R}$. This generalized Bessel function was further investigated in \cite{Ali-2, Ali-1} for $z \in \mathbb{D}= \{ z \in \mathbb{C}: |z|<1\}.$ It was shown in \cite{Ali-1} that the generalized Bessel function ${}_{a}\mathtt{B}_{b, p,c}$ is a solution of an $(a+1)$-order differential equation
 \begin{equation*}
 (D-p)\prod_{j=1}^a\left(D+\tfrac{2p+b+1-2j}{a}-p\right)y(x)+\frac{cx^2}{a^a 2^{1-a}}
y(x)= 0,
 \end{equation*}
 where the operator $D$ is given by $D := x (d/dx)$. For $a=1,$ the differential equation reduces to
$$x^2 y''(x) + bxy'(x) + (c x^2-p^2 +(1-b)p) y(x)=0.$$ Thus it yields the classical Bessel differential equation for $b=c=1.$ Interesting functional inequalities for ${}_a\mathtt{B}_{b, p, -\alpha^2}$ were obtained, particularly for the case $a=2.$

In \cite{BKS-ams}, Baricz \emph{et. al} investigated geometric properties involving the Bessel function of the first kind in $\mathbb{D}$ for the following three functions :
\begin{align}\notag
\mathtt{f}_{\nu}(z)&= \left(2^{\nu} \Gamma(\nu+1) \mathtt{J}_{\nu}(z)\right)^{\tfrac{1}{\nu}},\\
\mathtt{g}_{\nu}(z)&= 2^{\nu} \Gamma(\nu+1) z^{1-\nu} \mathtt{J}_{\nu}(z),\label{normal-BF-1}\\
\mathtt{h}_{\nu}(z)&= 2^{\nu} \Gamma(\nu+1) z^{1-\frac{\nu}{2}} \mathtt{J}_{\nu}(\sqrt{z}). \notag
\end{align}
 Each function is suitably normalized to ensure that it belongs to the class $\mathcal{A}$ consisting of analytic functions $f$ in $\mathbb{D}$ satisfying $f(0)=f'(0)-1=0$. Here the principal branch is assumed, which is positive for $z$ positive.

An important geometric feature of a complex-valued function is starlikeness. For $0 \leq \beta <1,$ the class of starlike functions of order $\beta$, denoted by $\mathcal{S}^{\ast}(\beta)$, are functions $f \in \mathcal{A}$ satisfying
 \begin{align*}
\real\left(\frac{z f'(z)}{f(z)}\right)>\beta \quad \text{for all} \quad z \in \mathbb{D}.
\end{align*}
In the case $\beta=0,$ these functions are simply said to be starlike (with respect to the origin). Geometrically $f \in \mathcal{S}^{\ast}:=\mathcal{S}^{\ast}(0)$  if the linear segment $tw$, $0 \leq t \leq 1,$ lies completely in $f(\mathbb{D})$ whenever $w \in f(\mathbb{D})$. A starlike function is necessarily univalent in $\mathbb{D}$.

The three functions given by \eqref{normal-BF-1} do not possess the property of starlikeness in the whole disk $\mathbb{D}$. Thus it is of interest to find the largest subdisk in $\mathbb{D}$ that gets mapped by these functions onto starlike domains. In general, the \emph{radius of starlikeness of order $\beta$} for a given class $\mathcal{G}$ of $\mathcal{A}$, denoted by $r^{*}_{\beta}$, is the largest number $r_0 \in (0,1)$ such that $r^{-1}f(rz)\in \mathcal{S}^{\ast}(\beta)$ for $0< r\leq r_0$ and for all $f \in \mathcal{G}$. Analytically,
\begin{align*}
r^{\ast}_\beta(\mathcal{G}):={\rm sup}\ \left\{r>0 :  \real\left(\frac{z f'(z)}{f(z)}\right)>\beta, \quad z \in \mathbb{D}_r, f \in \mathcal{G} \right\},
\end{align*}
where $\mathbb{D}_r= \{ z: |z|<r\}.$

In \cite{BKS-ams}, Baricz \emph{et. al} obtained the radius of starlikeness of order $\beta$ for each of the three functions $\mathtt{f}_{\nu}$, $\mathtt{g}_{\nu},$ and $\mathtt{h}_{\nu}$ given by \eqref{normal-BF-1}. This extends the earlier work of Brown in \cite{Brown-1} who obtained the radius of starlikeness (of order $0$) for functions $\mathtt{f}_{\nu}$ and $\mathtt{g}_{\nu}$.

For $a \in\mathbb{N},$ we consider the following extension of the three functions in \eqref{normal-BF-1} involving the generalized Bessel function:
\begin{align}\notag
\mathtt{f}_{a, \nu}(z)&:= \bigg(2^{a \nu-a+1} a^{-\frac{a(a\nu-a+1)}{2}} \Gamma(a \nu+1) {}_a\mathtt{B}_{2a-1, a \nu-a+1, 1}(a^{a/2} z)\bigg)^{\tfrac{1}{a \nu-a+1}},\\
\mathtt{g}_{a, \nu}(z)&:= 2^{a \nu-a+1} a^{-\frac{a}{2}(a\nu-a+1)} \Gamma(a \nu+1) z^{a-a\nu} {}_a\mathtt{B}_{2a-1, a \nu-a+1, 1}(a^{a/2} z),\label{normal-GBF-1}\\
\mathtt{h}_{a, \nu}(z)&:= 2^{a \nu-a+1} a^{-\frac{a}{2}(a\nu-a+1)} \Gamma(a \nu+1) z^{\frac{1}{2}(1+a-a\nu)} {}_a\mathtt{B}_{2a-1, a \nu-a+1, 1}(a^{a/2} \sqrt{z}). \notag
\end{align}
Here the function $\mathtt{f}_{a, \nu}$ is taken to be the principal branch (see the following section). Evidently for $a=1,$ these functions are those given by \eqref{normal-BF-1} treated by Baricz \emph{et. al} in \cite{BKS-ams}. Denote by $r_\beta^*(f)$ to be the radius of starlikeness of order $\beta$ for a given function ${f}$. In this paper, we find $r_\beta^*(f_a)$ when $f_a$ is either one of the three functions in \eqref{normal-GBF-1}. The best range on $\nu$ is also obtained for a fixed $a$ to ensure the functions $\mathtt{f}_{a, \nu}$ and $\mathtt{g}_{a, \nu}$ are starlike of order $\beta$ in the unit disk. When $a=1,$ the results obtained reduced to earlier known results.


For $a=1$, the generalized Bessel function \eqref{eqn:bessel-definition} is simply written as $\mathtt{B}_{b, p, c} := {}_1\mathtt{B}_{b, p, c}.$ Thus
\begin{equation}\label{Bessela1}
\mathtt{B}_{b, p, c}(z):= \sum_{k=0}^\infty \frac{(-c)^k}{k! \; \mathrm{\Gamma}{\left(k +p+\frac{b+1}{2}\right)}} \left(\frac{z}{2}\right)^{2k+p}.
\end{equation}

\begin{proposition}\cite[Proposition 2.2]{Ali-1}\label{product}
Let $a \in\mathbb{N},$ and $b, p, c, \in \mathbb{R}.$ Then
\begin{equation*}
{}_a \mathtt{B}_{b, p, c}(z)=(2 \pi)^{\tfrac{a-1}{2}}a^{-p-\tfrac{b}{2}}\left(\frac{z}{2}\right)^{p} \prod_{j=1}^{a}\left(\frac{z}{2 a^{a/2}}\right)^{-\frac{p+j-1}{a}}\mathtt{B}_{\frac{b+1-a}{a}, \frac{p+j-1}{a}, c}\left(\frac{z}{ a^{a/2}}\right),
\end{equation*}
where $\mathtt{B}_{b, p, c}$ is given by \eqref{Bessela1}.
\end{proposition}


In \cite{Ali-1}, the generalized Bessel function was also shown to satisfy the following relations:
\begin{equation*}\label{eqn:recurrence-re-2}
z \frac{d}{dz}{}_{a}\mathtt{B}_{b, p,c}(z) = p { }_{a}\mathtt{B}_{b, p, c}(z)
-c \left(\frac{z}{2}\right)^{1-a} z { }_{a}\mathtt{B}_{b, p+a, c}(z),
\end{equation*}
and
\begin{align}\label{eqn:recurrence-re-4}
 z\frac{d}{dz}{}_{a}\mathtt{B}_{b, p,c}(z) =\frac{z}{a}{}_{a}\mathtt{B}_{b, p-1,c}(z) -\left(\frac{2p+b-1}{a}- p\right) {}_{a}\mathtt{B}_{b, p,c}(z),
\end{align}
which together lead to the following result.
\begin{proposition}\cite[Proposition 2.3]{Ali-1}\label{recurrence}
Let $a \in\mathbb{N},$ $b, p, c \in \mathbb{R}$ and $ z \in \mathbb{D}$. Then
\begin{equation*}\label{eqn:prop1}
\frac{z}{a}{}_{a}\mathtt{B}_{b, p-1,c}(z)+ c \left(\frac{z}{2}\right)^{1-a} z { }_{a}\mathtt{B}_{b, p+a, c}(z)=\left(\tfrac{2p+b-1}{a}\right){}_a\mathtt{B}_{b, p,c}(z).
\end{equation*}
\end{proposition}

\section{Radius of starlikeness of generalized Bessel functions}

The following preliminary result sheds insights into the zeros of the three functions given by \eqref{normal-GBF-1}.

\begin{theorem}\label{zeros}
 Let $ \nu > (a-1)/a$, $a \in \mathbb{N}$. Then all zeros of ${}_a\mathtt{B}_{2a-1, a\nu-a+1, 1}(a^{{a}/{2}} z)$ are real. Further the origin is the only zero of ${}_a\mathtt{B}_{2a-1, a\nu-a+1, 1}(a^{{a}/{2}} z)$ in the unit disk $\mathbb{D}$.
\end{theorem}

\begin{proof}
Proposition \ref{product} shows that
\begin{align*}
{}_a \mathtt{B}_{2a-1, a\nu-a+1, 1}(a^{{a}/{2}} z)
&=(2 \pi)^{\tfrac{a-1}{a}}a^{-(a\nu+\tfrac{1}{2})}a^{\tfrac{a}{2}(a\nu-a+1)}\left(\frac{z}{2}\right)^{a\nu-a+1}\\ & \quad \quad \times \prod_{j=1}^{a}\left(\frac{z}{2}\right)^{-(\nu-1)-\tfrac{j}{a}} \mathtt{B}_{1,(\nu-1)+j/a,1}(z).
\end{align*}
Since $$\mathtt{B}_{1,(\nu-1)+j/a,1}(z)=\mathtt{J}_{(\nu-1)+j/a}(z),$$
it readily follows that
\begin{align*}
{}_a \mathtt{B}_{2a-1, a\nu-a+1, 1}(a^{{a}/{2}} z)
&=(2 \pi)^{\tfrac{a-1}{a}}a^{\tfrac{1}{2}(a^2\nu-2a\nu-a^2+a-1)}
\left(\frac{z}{2}\right)^{-\tfrac{1}{2}(a-1)}\\
& \quad \quad \times \mathtt{J}_{(\nu-1)+1/a}(z) \mathtt{J}_{(\nu-1)+2/a}(z) \ldots \mathtt{J}_{\nu}(z).
\end{align*}

Now, $\nu-1+(j/a) \geq \nu-1+(1/a) > 0$, $j=1, \ldots, a$. Further for $p>-1$, it is known \cite[p. 483]{Watson} that the zeros of $\mathtt{J}_p$ are all real. If $\mathtt{j}_{p,k}$ denotes the $k$-th positive zero of $\mathtt{J}_p$, it is also known \cite[p. 508]{Watson} that when $p$ is positive, the positive zeros of $\mathtt{J}_p$ increases as $p$ increases. Thus we infer that the zeros of ${}_a \mathtt{B}_{2a-1, a\nu-a+1, 1}(a^{{a}/{2}} z)$ are all real. Since
\[\mathtt{j}_{\nu,1} > \mathtt{j}_{\nu-1,1}> \ldots > \mathtt{j}_{\nu-1+(1/a), 1} >\mathtt{j}_{0, 1} \approx 2.40483,\]
the only zero in  $\mathbb{D}$ occurs at the origin.
\end{proof}

Theorem \ref{zeros} shows that the function
 \begin{equation*}
\mathtt{f}_{a, \nu}(z) = z \left(1+\sum_{k=1}^{\infty}\frac{(-1)^k \mathrm{\Gamma}{(a\nu+1)}a^{ak}}{k!2^{2k}\mathrm{\Gamma}{(ak+a\nu+1)}}z^{2k}
\right)^{\tfrac{1}{a\nu-a+1}}
\end{equation*}
have only one zero inside $\mathbb{D}$ whenever $\nu-1+(1/a) > 0$. Thus in this instance, we may take the principal branch for $\mathtt{f}_{a, \nu}  \in \mathcal{A}$. It is also readily verified that the functions $\mathtt{g}_{a, \nu}$ and
\begin{align*}
\mathtt{h}_{a, \nu}(z) &= z - \frac{\mathrm{\Gamma}{(a\nu+1)}}{1!2^2\mathrm{\Gamma}{(a+a\nu+1)}}a^az^2 + \frac{\mathrm{\Gamma}{(a\nu+1)}}{2!2^4\mathrm{\Gamma}{(2a+a\nu+1)}} a^{2a}z^3 + \cdots\\
& \quad \quad + (-1)^k\frac{\mathrm{\Gamma}{(a\nu+1)}}{k!2^{2k}\mathrm{\Gamma}{(ak+a\nu+1)}}a^{ak}z^{k+1}+\cdots
\end{align*}
are both analytic and belong to the normalized class $\mathcal{A}$.

The following is another preliminary result required in the sequel.

\begin{lemma}\label{lemma-1} Let $ a \in \mathbb{N}$ and $\nu >-1/a$. Then
\begin{align*}
\frac{z\;{}_{a}\mathtt{B}'_{2a-1, a \nu-a+1,1}(a^{a/2} z)}{{}_{a}\mathtt{B}_{2a-1, a \nu-a+1,1}(a^{a/2} z)}
=\frac{z\mathtt{J}_{\nu-1}\left(z\right)}
{\mathtt{J}_{\nu}\left(z\right)}-(2-a)\nu +1-a.
\end{align*}
\end{lemma}

\begin{proof}
Since
\begin{align*}
 \mathtt{B}_{b, p, 1}(z)=\left(\frac{2}{z}\right)^{\tfrac{b-1}{2}} \mathtt{J}_{p+\tfrac{b-1}{2}}(z),
\end{align*}
it follows from Proposition \ref{product} that
\begin{equation*}\label{special-case-1}
{}_a \mathtt{B}_{2a-1, a\nu-a+1, 1}(z)=(2 \pi)^{\tfrac{a-1}{a}}a^{-\tfrac{2 a \nu+1}{2}}\left(\frac{z}{2}\right)^{a\nu-a+1} \prod_{j=1}^{a}a^{\tfrac{a\nu-a+j}{2}}\left(\frac{2}{z}\right)^{\tfrac{a \nu-a+j}{a}}
\mathtt{J}_{\frac{a\nu-a+j}{a}}\left(\frac{z}{ a^{a/2}}\right),
\end{equation*}
and
\begin{equation*}\label{special-case-2}
{}_a \mathtt{B}_{2a-1, a\nu-a, 1}(z)=(2 \pi)^{\tfrac{a-1}{a}}a^{-\tfrac{2 a \nu-1}{2}}\left(\frac{z}{2}\right)^{a\nu-a} \prod_{j=1}^{a}a^{\tfrac{a\nu-a+j-1}{2}}\left(\frac{2}{z}\right)^{\tfrac{a \nu-a+j-1}{a}}
\mathtt{J}_{\frac{a\nu-a+j-1}{a}}\left(\frac{z}{ a^{a/2}}\right).
\end{equation*}
Expanding the above products, a routine calculation shows that
\begin{equation*}\label{special-case-ratio}
\frac{{}_a \mathtt{B}_{2a-1, a\nu-a, 1}(z)}{{}_a \mathtt{B}_{2a-1, a\nu-a+1, 1}(z)}
=a^{1-\tfrac{a}{2}} \frac{\mathtt{J}_{\nu-1}\left(\frac{z}{a^{a/2}}\right)}
{\mathtt{J}_{\nu}\left(\frac{z}{a^{a/2}}\right)}.
\end{equation*}

With $b=2a-1$ and $p=a \nu-a+1$, the recurrence relation \eqref{eqn:recurrence-re-4} gives
\begin{align*}\label{recurence-special}
 z\;\frac{d}{dz}{}_{a}\mathtt{B}_{2a-1, a \nu-a+1,1}(z) =\frac{z}{a}{}_{a}\mathtt{B}_{2a-1, a \nu-a,1}(z) -\left(\nu(2-a)+a-1\right) {}_{a}\mathtt{B}_{2a-1, a \nu-a+1,1}(z).
\end{align*}
Replacing $z$ by $a^{a/2} z$ leads to
\begin{align*}
\frac{z\;{}_{a}\mathtt{B}'_{2a-1, a \nu-a+1,1}(a^{a/2} z)}{{}_{a}\mathtt{B}_{2a-1, a \nu-a+1,1}(a^{a/2} z)}
&= \frac{a^{a/2} z}{a}\frac{{}_{a}\mathtt{B}_{2a-1, a \nu-a,1}(a^{a/2} z)}{{}_{a}\mathtt{B}_{2a-1, a \nu-a+1,1}(a^{a/2} z)} -(2-a)\nu+1-a\\
&=\frac{z\mathtt{J}_{\nu-1}\left(z\right)}
{\mathtt{J}_{\nu}\left(z\right)}-(2-a)\nu+1-a,
\end{align*}
which proves the assertion.
\end{proof}

A result on the modified Bessel function of order $p$ given by
 \begin{equation*}\label{modBessel}
\mathtt{I}_p(z) =  \sum_{k=0}^\infty \frac{1}{k! \; \mathrm{\Gamma}{\left(k +p+1\right)} } \left(\frac{z}{2}\right)^{2k+p}
\end{equation*}
is the final preliminary result required in the sequel.

\begin{proposition}\label{prop-1} Let $\alpha, \nu\in \mathbb{R}$ satisfy $-1<\nu < -\alpha$. Then the equation $r \mathtt{I}_{\nu}'(r)+\alpha \mathtt{I}_{\nu}(r)=0$ has a unique root in $(0, \infty)$.
\end{proposition}

\begin{proof}
Consider the function
\[q(r):= \frac{r \mathtt{I}_{\nu}'(r)}{\mathtt{I}_{\nu}(r)}+\alpha.\]
It is known from \cite[Theorem 3.1(c)]{Ali-1} that $r \mathtt{I}_{\nu}'(r)/I_{\nu}(r)$ is increasing on $(0, \infty)$. Further, the asymptotic properties show that $r \mathtt{I}_{\nu}'(r)/\mathtt{I}_{\nu}(r) \to \nu$ as $r \to 0$, and  $r \mathtt{I}_{\nu}'(r)/\mathtt{I}_{\nu}(r) \to \infty$ as $r \to \infty$. This implies that
$q(r) \to \nu+\alpha< 0$ as  $r \to 0$, and   $q(r) \to \infty$ for $r \to \infty$. Thus $q$ has exactly one zero.
\end{proof}

We also recall additional facts on the zeros of the Dini functions.

\begin{lemma}\cite[p. 482]{Watson}\label{lemma-2} If $\nu>-1$ and $\alpha, \gamma \in \mathbb{R}$, then the Dini function
$z \mapsto \alpha \mathtt{J}_\nu(z)+ \gamma z \mathtt{J}'_\nu(z)$ has all its zeros real whenever $((\alpha/\gamma)+ \nu) \geq 0$. In the case $((\alpha/\gamma)+ \nu) < 0,$ it also has two purely imaginary zeros.
\end{lemma}


\begin{lemma}\cite[Theorem 6.1]{IM-2}\label{lemma-4} Let $ \alpha \in \mathbb{R}$, $\nu>-1$ and $\nu +\alpha > 0$. Further let $x_{\nu, 1}$ be the smallest positive root of $\alpha \mathtt{J}_\nu(z)+ z \mathtt{J}'_\nu(z)=0$. Then $x_{\nu, 1}^2<\mathtt{j}_{\nu, 1}^2$.
\end{lemma}

\begin{lemma} \cite[p. 78]{IM-1}\label{lemma-5}  Let $-1<\nu <-\alpha$, and $\pm i \zeta$ be the single pair of conjugate purely imaginary zeros of the Dini function $z \mapsto \alpha \mathtt{J}_\nu(z)+ z \mathtt{J}'_\nu(z).$ Then
\[ \zeta^2 <-\frac{\alpha+\nu}{2+\alpha+\nu}\mathtt{j}^2_{\nu, 1}.\]
\end{lemma}

We are now ready to present the radius of starlikeness for each function given in \eqref{normal-GBF-1}.

\begin{theorem}\label{thm-1}
Let $0 \leq \beta <1$, and $a \in \mathbb{N}$. If $\nu > (a-1)/a$, then
$r_\beta^*(f_{a, \nu})= \mathtt{j}_{\nu, \beta,1}^{\mathtt{a, f}}$, where
$\mathtt{j}_{\nu, \beta,1}^{\mathtt{a, f}}$ is the smallest positive root of the equation
\begin{align}\label{eqn1a}
r a^{a/2} \mathtt{J}'_\nu(r)-\left( (\nu-1)(1-a)a^{a/2}+\beta(a \nu-a+1)\right) \mathtt{J}_\nu(r)=0.
\end{align}
If $\nu \in (-1/a, (a-1)/a)$ and
\begin{align}\label{asum} \frac{ (a \nu-a+1)\left(a^{a/2}-\beta\right)}{2 a^{a/2} +(a \nu-a+1)\left(a^{a/2}-\beta\right)} >-1,\end{align}
then $r_\beta^*(f_{a, \nu})=  \mathtt{i}_{\nu, \beta}^{\mathtt{a, f}}$, where $\mathtt{i}_{\nu, \beta}^{\mathtt{a, f}}$ is the unique positive  root of the equation
\begin{align}\label{eqn-f-3}
r a^{a/2} \mathtt{I}'_\nu(r)-\left( (\nu-1)(1-a)a^{a/2}+\beta(a \nu-a+1)\right) \mathtt{I}_\nu(r)=0.
\end{align}
\end{theorem}

\begin{proof}
Differentiating logarithmically, Lemma \ref{lemma-1} shows that
\begin{align}\label{eqn:6}\notag
\frac{z \mathtt{f}'_{a, \nu}(z)}{\mathtt{f}_{a, \nu}(z)}&=\frac{a^{a/2}}{a\nu-a+1}
\frac{z\;{}_{a}\mathtt{B}'_{2a-1, a \nu-a+1,1}(a^{a/2} z)}{{}_{a}\mathtt{B}_{2a-1, a \nu-a+1,1}(a^{a/2} z)}\\
&= \frac{a^{a/2}}{a\nu-a+1} \left(\frac{z\mathtt{J}_{\nu-1}\left(z\right)}
{\mathtt{J}_{\nu}\left(z\right)}-\nu(2-a)+1-a\right).
\end{align}
Since ${}_1\mathtt{B}_{1, \nu, 1}( z)=\mathtt{J}_{\nu}(z)$, the relation \eqref{eqn:recurrence-re-4} leads to  the well-known recurrence relation
\begin{align*}
z \mathtt{J}'_{\nu}(z)= z \mathtt{J}_{\nu-1}(z)-\nu \mathtt{J}_{\nu}(z),
\end{align*}
and whence \eqref{eqn:6} reduces to
\begin{align}\label{eqn:7}
\frac{z \mathtt{f}'_{a, \nu}(z)}{\mathtt{f}_{a, \nu}(z)}
= \frac{a^{a/2}}{a\nu-a+1} \left(\frac{z\mathtt{J}'_{\nu}(z)}
{\mathtt{J}_{\nu}(z)}-(\nu-1)(1-a)\right).
\end{align}

With $\mathtt{j}_{\nu, n}$ as the n-th positive zero of the Bessel
function $\mathtt{J}_{\nu}$, the Bessel function $\mathtt{J}_{\nu}$ admits the Weierstrassian decomposition \cite[p.498]{Watson}
\begin{align*}
\mathtt{J}_{\nu}(z)=\frac{z^\nu}{2^{\nu}\Gamma(\nu+1)} \prod_{n=1}^\infty \left(1-\frac{z^2}{\mathtt{j}^2_{\nu, n}}\right).
\end{align*}
Thus
\begin{align*}
\frac{z\;\mathtt{J}'_{\nu}(z)}{\mathtt{J}_{\nu}(z)}=\nu-\sum_{n=1}^\infty \frac{2z^2}{\mathtt{j}^2_{\nu, n}-z^2},
\end{align*}
which reduces \eqref{eqn:7} to
\begin{align}\label{eqn:8}
\frac{z \mathtt{f}'_{a, \nu}(z)}{\mathtt{f}_{a, \nu}(z)}
&=a^{a/2} -\frac{a^{a/2}}{a\nu-a+1}\sum_{n=1}^\infty \frac{2z^2}{\mathtt{j}^2_{\nu, n}-z^2}.
\end{align}

For $\nu > (a-1)/a$ and $|z| < \mathtt{j}_{\nu, n}$, evidently
\begin{align*}
\real\frac{z \mathtt{f}'_{a, \nu}(z)}{\mathtt{f}_{a, \nu}(z)}
&=a^{a/2} -\frac{a^{a/2}}{a\nu-a+1}\real\sum_{n=1}^\infty \frac{2z^2}{\mathtt{j}^2_{\nu, n}-z^2}\\
&\geq a^{a/2} -\frac{a^{a/2}}{a\nu-a+1}\sum_{n=1}^\infty \frac{2|z|^2}{\mathtt{j}^2_{\nu, n}-|z|^2}=\frac{|z|  \mathtt{f}'_{a, \nu}(|z|)}{\mathtt{f}_{a, \nu}(|z|)}.
\end{align*}
Equality holds for $|z|=r$, and by the minimum principle for harmonic functions,
\[\real\frac{z \mathtt{f}'_{\nu}(z)}{\mathtt{f}_{\nu}(z)} \geq \beta \iff |z|\leq\mathtt{j}_{\nu, \beta,1}^{\mathtt{a, f} }, \]
where $\mathtt{j}_{\nu, \beta,1}^{\mathtt{a, f}}$ is the smallest positive root of equation \eqref{eqn1a}.
Since
\[ \nu-\left((\nu-1) (1-a)+\tfrac{\beta( a \nu-a+1)}{a^{a/2}}\right)
= ( a \nu-a+1)\left(1-\tfrac{\beta}{a^{a/2}}\right) > 0\]
for all $\nu >(a-1)/a$, we infer from Lemma \ref{lemma-2} and Lemma \ref{lemma-4} that $\mathtt{j}_{\nu, \beta,1}^{\mathtt{a, f} }< \mathtt{j}_{\nu, 1} <  \mathtt{j}_{\nu, n}$.

Consider next the case $-1/a < \nu< (a-1)/a$. It is known from \cite[p. 2023]{BKS-ams} that for $z \in \mathbb{C}$ and $\alpha \in \mathbb{R}$  with $\alpha \geq |z|$, then
\[\real\left(\frac{z}{\alpha-z}\right)\geq -\frac{|z|}{\alpha+|z|},\]
which in turn implies that
\[\real\left(\frac{z^2}{\mathtt{j}_{\nu,n}^2-z^2}\right)\geq -\frac{|z|^2}{\mathtt{j}_{\nu,n}^2+|z|^2}\]
whenever $|z|<\mathtt{j}_{\nu,1}<\mathtt{j}_{\nu,n}$.

The expression \eqref{eqn:8} yields
\begin{align*}
\real\frac{z \mathtt{f}'_{a, \nu}(z)}{\mathtt{f}_{a, \nu}(z)}
&\geq a^{a/2} +\frac{a^{a/2}}{a\nu-a+1}\sum_{n=1}^\infty \frac{2|z|^2}{\mathtt{j}^2_{\nu, n}+|z|^2}=\frac{i|z|  \mathtt{f}'_{a, \nu}(i|z|)}{\mathtt{f}_{a, \nu}(i|z|)}.
\end{align*}
Equality holds for $|z|=i |z|= ir$. Hence
\[ \real\frac{z \mathtt{f}'_{a, \nu}(z)}{\mathtt{f}_{a, \nu}(z)} \geq \beta\]
if $|z|\leq \mathtt{i}_{\nu, \beta}^{a, f}$, where $\mathtt{i}_{\nu, \beta}^{a, f}$ is
a root of $i|z|  \mathtt{f}'_{a, \nu}(i|z|)=\beta \mathtt{f}_{a, \nu}(i|z|),$
 that is, $\mathtt{i}_{\nu, \beta}^{a, f}$ is a root of
 \[ \frac{a^{a/2}}{a\nu-a+1} \left(\frac{i |z| \mathtt{J}'_\nu(i| z|)}{\mathtt{J}_\nu(i |z|)}-(\nu-1)(1-a)\right) =\beta.\]
Since $\mathtt{I}_{\nu}(z)=i^{-\nu} \mathtt{J}_{\nu}(i z)$, the above equation is equivalent to \eqref{eqn-f-3}. It also follows from Proposition \ref{prop-1} that the root $\mathtt{i}_{\nu, \beta}^{a, f}$ is unique. Finally, that $\mathtt{i}_{\nu, \beta}^{a, f}<\mathtt{j}_{\nu, n}$ is a consequence of Lemma \ref{lemma-5} and assumption \eqref{asum}. Indeed,
 \[\left(\mathtt{i}_{\nu, \beta}^{a, f}\right)^2 <-\frac{(a\nu-a+1)\left(a^{a/2}-\beta\right)}{2 a^{a/2}+(a\nu-a+1)\left(a^{a/2}-\beta\right)}\mathtt{j}^2_{\nu, 1}
 < \mathtt{j}^2_{\nu, 1} < \mathtt{j}^2_{\nu, n},\]
which completes the proof.
\end{proof}

Interestingly, Theorem \ref{thm-1} reduces to earlier known result for $a=1$.

\begin{corollary}\cite[Theorem 1(a)]{BKS-ams}
Let $0 \leq \beta <1$. If $\nu > 0$, then $r_\beta^*(f_{1, \nu})$ is the smallest positive root $\mathtt{j}_{\nu, \beta,1}^{\mathtt{1, f}}$ of the equation
\begin{align*}
r  \mathtt{J}'_\nu(r)-\beta\nu \mathtt{J}_\nu(r)=0.
\end{align*}
In the case $\nu \in (-1, 0)$, then $r_\beta^*(f_{1, \nu})$ is the unique positive root $\mathtt{i}_{\nu, \beta}^{\mathtt{1, f}}$ of the equation
\begin{align*}
r \mathtt{I}'_\nu(r)-\beta \nu \mathtt{I}_\nu(r)=0.
\end{align*}
\end{corollary}

The next two results find the radius of starlikeness of order $\beta$ for the functions $\mathtt{g}_{a, \nu}$ and $\mathtt{h}_{a, \nu}$ given in \eqref{normal-GBF-1}.

\begin{theorem}\label{theorem-g}
 Let $\beta \in [0,1)$, $a\in \mathbb{N}$, and  $\nu > -1/a$. If $a(\nu-1)(a^{a/2}-1)+a^{a/2}-\beta \geq 0$, then
$r_\beta^*(\mathtt{g}_{a, \nu})= \mathtt{j}_{\nu, \beta, 1}^{\mathtt{a, g}}$,  where
$\mathtt{j}_{\nu, \beta, 1}^{\mathtt{a, g}}$ is the smallest positive root of the equation
\begin{align}\label{eqn2a}
 r a^{a/2} \mathtt{J}'_\nu(r)-\left( (\nu-1)(1-a)a^{a/2}-a(1-\nu)+\beta \right) \mathtt{J}_\nu(r)=0.
\end{align}
\end{theorem}

\begin{proof}
It follows from \eqref{normal-GBF-1} that
\begin{align*}
\frac{z\mathtt{g}_{a, \nu}'(z)}{\mathtt{g}_{a, \nu}(z)}&=a (1-\nu)+ a^{a/2}  \frac{z\;{}_{a}\mathtt{B}'_{2a-1, a \nu-a+1,1}(a^{a/2} z)}{{}_{a}\mathtt{B}_{2a-1, a \nu-a+1,1}(a^{a/2} z)}.
\end{align*}
As in the proof of Theorem \ref{thm-1} (see \eqref{eqn:6} and \eqref{eqn:7}), it is readily shown that
\begin{align}\label{eqn-g} \notag
\frac{z\mathtt{g}_{a, \nu}'(z)}{\mathtt{g}_{a, \nu}(z)}&= a (1-\nu)+ a^{a/2} \left(\frac{z\mathtt{J}_{\nu-1}\left(z\right)}
{\mathtt{J}_{\nu}\left(z\right)}-\nu(2-a)+1-a\right)\\ \notag
&= a (1-\nu)+ a^{a/2}\left(\frac{z\mathtt{J}'_{\nu}\left(z\right)}
{\mathtt{J}_{\nu}\left(z\right)}-(\nu-1)(1-a)\right)\\
&= a (1-\nu)+ a^{a/2}\left[a\nu+1-a-\sum_{n=1}^\infty \frac{2z^2}{\mathtt{j}^2_{\nu, n}-z^2}\right].
\end{align}
This implies that
\[ \real \frac{z\mathtt{g}_{a, \nu}'(z)}{\mathtt{g}_{a, \nu}(z)}
\geq a (1-\nu)+ a^{a/2}\left[a\nu+1-a-\sum_{n=1}^\infty \frac{2|z|^2}{\mathtt{j}^2_{\nu, n}-|z|^2}\right]=\frac{|z|\mathtt{g}_{a, \nu}'(|z|)}{\mathtt{g}_{a, \nu}(|z|)}\]
provided $|z|<\mathtt{j}_{\nu,n}$. Equality holds at $|z|=r$. The minimum principle for harmonic functions leads to
\[ \real \frac{z\mathtt{g}_{a, \nu}'(z)}{\mathtt{g}_{a, \nu}(z)} \geq \beta \iff |z|\leq r_\beta^*(\mathtt{g}_{a, \nu}).\]

The exact value of $r_\beta^*(\mathtt{g}_{a, \nu})$ is obtained from
the equation
$r\mathtt{g}_{a, \nu}'(r)=\beta \mathtt{g}_{a, \nu}(r)$. From \eqref{eqn-g}, this is equivalent to determining the root of \eqref{eqn2a}.

If $a(\nu-1)(a^{a/2}-1)+a^{a/2}-\beta \geq 0$, then Lemma \ref{lemma-2} shows that all roots of \eqref{eqn2a} are real. In this case, $r_\beta^*(\mathtt{g}_{a, \nu})$ is its smallest positive  root $\mathtt{j}_{\nu, \beta, 1}^{a, g}$. Finally, Lemma \ref{lemma-4} shows that $\mathtt{j}_{\nu, \beta, 1}^{a, g}<\mathtt{j}_{\nu,  1},$  and whence $|z|<r_\beta^*(\mathtt{g}_{a, \nu})<\mathtt{j}_{\nu,  1}$.
\end{proof}

\begin{theorem}\label{theorem-h}
Let $\beta \in [0,1)$, $a\in \mathbb{N}$, and  $\nu > -1/a$. If $(a^{a/2}-1)(1-a+a\nu)+2(1-\beta)>0$, then
$r_\beta^*(\mathtt{h}_{a, \nu})= \mathtt{j}_{\nu, \beta, 1}^{\mathtt{a, h}}$,  where
$\mathtt{j}_{\nu, \beta, 1}^{\mathtt{a, h}}$ is the smallest positive root of the equation
\begin{align}\label{eqnh-1}
a^{a/2} r \mathtt{J}_\nu'(r) + \left( (a^{a/2}-1)(1-a+a\nu) -a^{a/2} \nu+2(1- \beta)\right)\mathtt{J}_\nu(r)=0.
\end{align}
\end{theorem}

\begin{proof}
It follows from \eqref{normal-GBF-1} that
\begin{align*}
\frac{\mathtt{h}_{a, \nu}'(z)}{\mathtt{h}_{a, \nu}(z)}= \frac{1+a-a \nu}{2z}+
\frac{ a^{a/2}}{2 \sqrt{z}} \frac{{}_a\mathtt{B}_{2a-1, a \nu-a+1, 1}'(a^{a/2} \sqrt{z})}
{{}_a\mathtt{B}_{2a-1, a \nu-a+1, 1}(a^{a/2} \sqrt{z})},
\end{align*}
and thus Lemma \ref{lemma-1} yields
\begin{align*}
\frac{z \mathtt{h}_{a, \nu}'(z)}{\mathtt{h}_{a, \nu}(z)}&= \frac{1+a-a \nu}{2}+
\frac{ a^{a/2} \sqrt{z} }{2 } \frac{{}_a\mathtt{B}_{2a-1, a \nu-a+1, 1}'(a^{a/2} \sqrt{z})}
{{}_a\mathtt{B}_{2a-1, a \nu-a+1, 1}(a^{a/2} \sqrt{z})}\\
& =\frac{1+a-a \nu}{2}+
\frac{ a^{a/2} }{2 } \left(\frac{\sqrt{z} \mathtt{J}_\nu'(\sqrt{z})}{\mathtt{J}_\nu(\sqrt{z})}-(\nu-1)(1-a)\right)\\
&=1-\frac{(a\nu+1-a)(1-a^{a/2})}{2}-a^{a/2} \sum_{n=1}^\infty \frac{z}{\mathtt{j}_{\nu, n}^2-z}.
\end{align*}

Proceeding similarly as in the proof of Theorem \ref{thm-1}, it is readily shown that
\begin{align*}
\real\frac{z \mathtt{h}_{a, \nu}'(z)}{\mathtt{h}_{a, \nu}(z)}
\geq 1-\frac{(a\nu+1-a)(1-a^{a/2})}{2}- a^{a/2} \sum_{n=1}^\infty \frac{|z|}{\mathtt{j}_{\nu, n}^2-|z|}=\frac{|z| \mathtt{h}_{a, \nu}'(|z|)}{\mathtt{h}_{a, \nu}(|z|)}=\beta
\end{align*}
if and only if $|z|\leq r^*(\mathtt{h}_{\nu, \beta}) < \mathtt{j}_{\nu,n}$. Here $r^*(\mathtt{h}_{\nu, \beta})$ is the smallest root of the equation $r \mathtt{h}_{a, \nu}'(r) / \mathtt{h}_{a, \nu}(r)=\beta,$ that is, a root of
\[\frac{1+a-a \nu}{2}+\frac{ a^{a/2} }{2 } \left(\frac{\sqrt{r} \mathtt{J}_\nu'(\sqrt{r})}{\mathtt{J}_\nu(\sqrt{r})}-(\nu-1)(1-a)\right)=\beta\\,\]
or equivalently, of the equation
\begin{align*}\label{h1}
a^{a/2} r \mathtt{J}_\nu'(r) + \left( (a^{a/2}-1)(1-a+a\nu) -a^{a/2} \nu+2(1-\beta)\right)\mathtt{J}_\nu(r)=0.
\end{align*}
Thus by Lemma \ref{lemma-2}, $r^*(\mathtt{h}_{\nu, \beta})$ is the smallest positive root $\mathtt{j}_{\nu, \beta, 1}^{a, h}$ of \eqref{eqnh-1} when $(a^{a/2}-1)(1-a+a\nu)+2(1-\beta)>0.$
\end{proof}

\begin{remark}
 In the case $a=1$, the condition $ a(\nu-1)\left(a^{a/2}-1\right) +a^{a/2}-\beta=1-\beta >0$ and $\left(a^{a/2}-1\right)(1-a+a\nu)+2(1-\beta)=2(1-\beta)>0$ both hold trivially for all $\beta \in [0,1)$.  Both theorems therefore coincide with the earlier results in \cite{BKS-ams}.

Further, it is of interest to determine the radius of starlikeness $r_\beta^*(\mathtt{g}_{a, \nu})$ in Theorem \ref{theorem-g} in the event that $a(\nu-1)(a^{a/2}-1)+a^{a/2}-\beta < 0$, as well as that of  $r_\beta^*(\mathtt{h}_{a, \nu})$ in Theorem \ref{theorem-h} when $(a^{a/2}-1)(1-a+a\nu)+2(1-\beta)<0$.
\end{remark}

\section{Starlikeness of the generalized Bessel function}
In this final section, the best range on $\nu$ is obtained for a fixed $a \in \mathbb{N}$ to ensure the functions $\mathtt{f}_{a, \nu}$ and $\mathtt{g}_{a, \nu}$ given by \eqref{normal-GBF-1} are starlike of order $\beta$ in $\mathbb{D}$.

\begin{theorem}\label{thm-f-star}
For a fixed $a \in \mathbb{N},$ the function $\mathtt{f}_{a, \nu}$ given by \eqref{normal-GBF-1} is starlike of order $\beta \in [0, 1)$ in $\mathbb{D}$ if and only if $\nu \geq \nu_{f}(a,\beta)$, where $\nu_{f}(a, \beta)$ is the unique root of $$(a\nu-a+1)(a^{a/2}-\beta)\mathtt{J}_\nu(1)=a^{a/2} \mathtt{J}_{\nu+1}(1)$$ in $( (a-1)/a, \infty)$.
\end{theorem}

\begin{proof} For $\nu >(a-1)/a$ and $|z|=r \in [0,1),$ it follows from \eqref{eqn:8} that
\begin{align*}
\real\left(\frac{ z \mathtt{f}'_{a, \nu}(z)}{\mathtt{f}_{a, \nu}(z)}\right)\geq \frac{ r
\mathtt{f}'_{a, \nu}(r)}{ \mathtt{f}_{a, \nu}(r)}=
a^{a/2} -\frac{  a^{a/2}}{a \nu-a+1} \sum_{n=1}^\infty \frac{ 2 r^2}{\mathtt{j}^2_{\nu, n}-r^2}.
\end{align*}
The above inequality holds since $r<1$ and it is known (\cite[p. 508]{Watson}, \cite[p. 236]{Olver}) that  the function $\nu \mapsto \mathtt{j}_{\nu, n}$ is increasing on $(0,\infty)$ for each fixed $n \in \mathbb{N}$, and whence
$\mathtt{j}_{\nu, 1} \geq \mathtt{j}_{(a-1)/a, 1} \geq \mathtt{j}_{0, 1} \approx  2.40483\ldots$.

A computation yields
\[ \frac{d}{dr}\left(\frac{ r
\mathtt{f}'_{a, \nu}(r)}{ \mathtt{f}_{a, \nu}(r)}\right)=-\frac{ 2 a^{a/2}}{a \nu-a+1} \sum_{n=1}^\infty \frac{2 r \mathtt{j}_{\nu, n}^2}{\left(\mathtt{j}_{\nu, n}^2-r^2\right)^2} \leq 0. \]
Hence
\begin{align*}
\real\left(\frac{ z \mathtt{f}'_{a, \nu}(z)}{\mathtt{f}_{a, \nu}(z)}\right)\geq
a^{a/2} -\frac{  a^{a/2}}{a \nu-a+1} \sum_{n=1}^\infty \frac{ 2 }{\mathtt{j}^2_{\nu, n}-1}= \frac{\mathtt{f}'_{a, \nu}(1)}{\mathtt{f}_{a, \nu}(1)}.
\end{align*}

The monotonicity property of $\nu \mapsto \mathtt{j}_{\nu, n}$ leads to
 \[\frac{\mathtt{f}'_{a, \mu}(1)}{\mathtt{f}_{a, \mu}(1)}=a^{a/2} -\frac{  a^{a/2}}{a \mu-a+1} \sum_{n=1}^\infty \frac{ 2 }{\mathtt{j}^2_{\mu, n}-1}
 \ge a^{a/2} -\frac{  a^{a/2}}{a \nu-a+1} \sum_{n=1}^\infty \frac{ 2 }{\mathtt{j}^2_{\nu, n}-1}=\frac{\mathtt{f}'_{a, \nu}(1)}{\mathtt{f}_{a, \nu}(1)},\]
 $\mu \geq \nu>-1$. Since $\nu \mapsto {\mathtt{f}'_{a, \nu}(1)}/{\mathtt{f}_{a, \nu}(1)}$ is increasing in $\left((a-1)/a, \infty\right),$ and from consideration of the asymptotic behavior of ${\mathtt{f}'_{a, \nu}(1)}/{\mathtt{f}_{a, \nu}(1)},$ evidently ${\mathtt{f}'_{a, \nu}(1)}/{\mathtt{f}_{a, \nu}(1)}\geq \beta$ if and only if $\nu \geq \nu_f(a,\beta),$ where  $\nu_f(a,\beta)$ is the unique root of the equation $\mathtt{f}'_{a, \nu}(1)=\beta \mathtt{f}_{a, \nu}(1).$ From \eqref{eqn:6}, the latter equation is equivalent to
\[ a^{a/2} \mathtt{J}_{\nu-1}(1)=\left( a^{a/2}(\nu(2-a)+a-1)+\beta(a\nu-a+1)\right)\mathtt{J}_{\nu}(1).\]
The recurrence relation in Proposition \ref{recurrence} now shows that $\nu_f(a, \beta)$ is a unique root of $(a\nu-a+1)(a^{a/2}-\beta)\mathtt{J}_\nu(1)=a^{a/2} \mathtt{J}_{\nu+1}(1)$. Since all inequalities are sharp, it follows that the value $\nu_f(a, \beta)$ is best.
\end{proof}

\begin{remark}
With regards to Theorem \ref{thm-f-star}, we tabulate the best value $\nu$ for a fixed $\beta$ and $a$ for which $\mathtt{f}_{a, \nu}$ is starlike of order $\beta$. These values are given in Table \ref{table-1}.

\begin{table}[h]
  \centering
  \begin{tabular}{|c|c|c|c|}
    \hline
     & $\beta=0$ & $\beta=0.5$& $\beta=0.95$\\ \hline
    $a=1$ & $\nu=0.39001$ & $\nu=0.645715$ & $\nu=2.72421$ \\ \hline
    $a=2$ & $\nu=0.659908$ & $\nu=0.706779$ &$\nu=0.781815$ \\ \hline
    $a=3$ & $\nu=0.766251$ & $\nu=0.776181$  &$\nu=0.786989$\\
    \hline
  \end{tabular}
  \vspace{.2in}
  \caption{Values of $\nu$ for $\mathtt{f}_{a, \nu}$ to be starlike} \label{table-1}
\end{table}

Using \eqref{eqn1a}, we tabulate the radius of starlikeness for $\mathtt{f}_{a, \nu}$ in Theorem \ref{thm-1} for a fixed $\nu=0.7$, $a=1,2,3$, and respectively $\beta=0$ and $\beta=1/2$. These are given in Table \ref{table-2}. Here the value of $\mathtt{j}_{\nu,1}$ at  $\nu=0.7$ is $\mathtt{j}_{0.7,1}=3.42189$. With reference to Table \ref{table-1}, we expect the radius of starlikeness to be less than $1$ whenever $\nu=0.7$ is less than the given values of $\nu$ in Table \ref{table-1}.

\begin{table}[h]
  \centering
  \begin{tabular}{|c|c|c|c|}
    \hline
     & $\beta=0$ & $\beta=1/2$ & $\beta=0.95$\\ \hline
    $a=1$ & $r_0^*(f_{1, 0.7})=1.44678$ & $r_{1/2}^*(f_{1, 0.7})=1.05621$& $r_{0.95}^*(f_{1, 0.7})=0.343848$\\ \hline
    $a=2$ & $r_0^*(f_{2, 0.7})=1.12397$ & $r_{1/2}^*(f_{2, 0.7})=0.982365$& $r_{0.95}^*(f_{2, 0.7})=0.828745$ \\ \hline
    $a=3$ & $r_0^*(f_{3, 0.7})=0.577726$ & $r_{1/2}^*(f_{3, 0.7})=0.549716$& $r_{0.95}^*(f_{3, 0.7})=0.523133$ \\
    \hline
  \end{tabular}
  \vspace{.2in}
  \caption{Radius of starlikeness for $\mathtt{f}_{a, \nu}$ when $\nu=0.7$} \label{table-2}
\end{table}

Letting
\[F(r):= \frac{r \mathtt{J}_{\nu}'(r)}{\mathtt{J}_{\nu}(r)},\]
then \eqref{eqn1a} takes the form $F(r)=-\alpha,$ where
\[\alpha := \alpha(a, \beta,\nu)= -\left((\nu-1) (1-a)+\frac{\beta( a \nu-a+1)}{a^{a/2}}\right).\]
For $\nu > 0$, it is known \cite{BMPon} that F(r) is strictly decreasing on $(0,\infty)$ except at the zeros of $\mathtt{J}_{\nu}(r)$. Differentiating with respect to $\beta$, it is clear that $\alpha$ is decreasing with respect to $\beta$ so long as $a \nu-a+1>0,$ and thus $r_{\beta}^*$ is decreasing. Further, for a fixed $\nu<1$ and $\beta=0$, then $\alpha$ is monotonically decreasing with respect to $a$, that is, $r_{0}^*$ is decreasing as a function of $a$. However, for $\beta$ near $1$, then $\alpha$ is no longer monotonic. For instance, choosing $\nu=0.7$ and $\beta=0.95$, Table \ref{table-2} illustrates the fact that $r_{0.95}^*$ is not monotonic with respect to the parameter $a$.
\end{remark}

\begin{theorem}\label{thm-g-star}
Let $a \in \mathbb{N},$ $\nu >-1/a$, and $\mathtt{j}_{\nu, 1}$ be the first positive zero of $\mathtt{J}_{\nu}$. Then
the function $\mathtt{g}_{a, \nu}$ given by \eqref{normal-GBF-1} is starlike of order $\beta \in [0, 1)$ in $\mathbb{D}$ if and only if $\nu \geq \nu_{g}(a,\beta)$, where $\nu_{g}(a, \beta)$ is the unique root in $( \max\{\tilde{\nu},-1/a\}, \infty)$ of $$(a(\nu-1)(a^{a/2}-1)+a^{a/2}-\beta)\mathtt{J}_\nu(1)=a^{a/2} \mathtt{J}_{\nu+1}(1),$$ and $\tilde{\nu}\simeq -0.7745\ldots$ is the unique root of $\mathtt{j}_{\nu, 1}=1$.
\end{theorem}

\begin{proof}
It follows from \eqref{eqn-g} that
\begin{align*}
\frac{z\mathtt{g}_{a, \nu}'(z)}{\mathtt{g}_{a, \nu}(z)}
&= a (1-\nu)+ a^{a/2}\left(\frac{z\mathtt{J}'_{\nu}\left(z\right)}
{\mathtt{J}_{\nu}\left(z\right)}-(\nu-1)(1-a)\right)\\
&= a (1-\nu)+ a^{a/2}\left(a\nu-a+1-\sum_{n=1}^\infty \frac{ 2 z^2}{\mathtt{j}_{\nu, n}^2-z^2}\right),
\end{align*}
and
\begin{align*}
\real\left(\frac{z\mathtt{g}_{a, \nu}'(z)}{\mathtt{g}_{a, \nu}(z)}\right)
\geq \frac{r\mathtt{g}_{a, \nu}'(r)}{\mathtt{g}_{a, \nu}(r)}=  a (1-\nu)+ a^{a/2}\left(a\nu-a+1-\sum_{n=1}^\infty \frac{ 2 r^2}{\mathtt{j}_{\nu, n}^2-r^2}\right)
\end{align*}
for $|z|<\mathtt{j}_{\nu, 1}$. The function $r \mapsto
{r\mathtt{g}_{a, \nu}'(r)}/{\mathtt{g}_{a, \nu}(r)}$ is decreasing on $[0, 1).$ Since $\mathtt{j}_{\nu, 1 } > 1$ for $\nu > \max\{\tilde{\nu}, -1/a\}$, it follows that $\mathtt{j}_{\nu, n } > 1$ for each $n$, and consequently
\begin{align*}
\real\left(\frac{z\mathtt{g}_{a, \nu}'(z)}{\mathtt{g}_{a, \nu}(z)}\right)
\geq \frac{r\mathtt{g}_{a, \nu}'(r)}{\mathtt{g}_{a, \nu}(r)}> \frac{\mathtt{g}_{a, \nu}'(1)}{\mathtt{g}_{a, \nu}(1)}=  a (1-\nu)+ a^{a/2}\left(a\nu-a+1-\sum_{n=1}^\infty \frac{ 2 }{\mathtt{j}_{\nu, n}^2-1}\right).
\end{align*}

Further, for $\mu \geq \nu > \max\{\tilde{\nu},-1/a\}$,
\begin{align*}
\frac{\mathtt{g}_{a, \mu}'(1)}{\mathtt{g}_{a, \mu}(1)}&= a\left(a^{a/2}-1\right)(\mu-1)+a^{a/2}-\sum_{n=1}^\infty \frac{ 2  a^{a/2} }{\mathtt{j}_{\mu,
n}^2-1}\\
&\geq a\left(a^{a/2}-1\right)(\nu-1)+a^{a/2}-\sum_{n=1}^\infty \frac{ 2  a^{a/2} }{\mathtt{j}_{\nu,
n}^2-1}=\frac{\mathtt{g}_{a, \nu}'(1)}{\mathtt{g}_{a, \nu}(1)}.
\end{align*}
This implies that $\nu \mapsto \mathtt{g}_{a, \nu}'(1)/\mathtt{g}_{a, \nu}(1)$ is increasing on $( \max\{\tilde{\nu},-1/a\}, \infty)$.

Thus
\[ \real\left(\frac{z\mathtt{g}_{a, \nu}'(z)}{\mathtt{g}_{a, \nu}(z)}\right)
> \frac{\mathtt{g}_{a, \nu}'(1)}{\mathtt{g}_{a, \nu}(1)}\geq \beta\]
if and only if $\nu \geq \nu_g(a,\beta)$, where $\nu_g(a,\beta)$ is the unique root of $\mathtt{g}_{a, \nu}'(1)=\beta \mathtt{g}_{a, \nu}(1),$ or equivalently,
\[a (1-\nu)+ a^{a/2}\left(\frac{\mathtt{J}'_{\nu}\left(1\right)}
{\mathtt{J}_{\nu}\left(1\right)}-(\nu-1)(1-a)\right)=\beta.\]
Finally, Proposition \ref{recurrence} implies that $\nu_{g}(a, \beta)$ is a unique root of
$$(a(\nu-1)(a^{a/2}-1)+a^{a/2}-\beta)\mathtt{J}_\nu(1)=a^{a/2} \mathtt{J}_{\nu+1}(1).$$
\end{proof}

\begin{remark}
 The best value $\nu$ obtained from  Theorem \ref{thm-g-star} for a fixed $\beta$ and $a$ for which $\mathtt{g}_{a, \nu}$ is starlike of order $\beta$ is given in Table \ref{table-3}.

\begin{table}[h]
  \centering
  \begin{tabular}{|c|c|c|c|}
    \hline
     & $\beta=0$ & $\beta=0.5$ & $\beta=0.95$\\ \hline
    $a=1$ & $\nu=-0.340092$ & $\nu=0.122499$& $\nu=9.02272$ \\ \hline
    $a=2$ & $\nu=0.39002$ & $\nu=0.586273$ & $\nu=0.772587$ \\ \hline
    $a=3$ & $\nu=0.714616$ & $\nu=0.751407$& $\nu=0.784626$  \\
    \hline
  \end{tabular}
  \vspace{.2in}
  \caption{Values of $\nu$ for $\mathtt{g}_{a, \nu}$ to be starlike} \label{table-3}
\end{table}
The radius of starlikeness for $\mathtt{g}_{a, \nu}$ drawn from Theorem \ref{theorem-g} is tabulated in Table \ref{table-4} for a fixed $\nu=0.7$, $a=1,2,3$, and respectively $\beta=0$, $\beta=0.5$ and $\beta=0.95$. Here the radius of starlikeness is expectedly less than $1$ whenever $\nu=0.7$ is less than the given values of $\nu$ in Table \ref{table-3}. A similar situation occurs as for the function $\mathtt{f}_{a, \nu}$ with regard to the monotonicity of the radius of starlikeness with respect to either parameter $\beta$ or $a$.

\begin{table}[h]
  \centering
  \begin{tabular}{|c|c|c|c|c|}
    \hline
     & $\beta=0$ & $\beta=1/2$ &$\beta=0.95$ \\ \hline
    $a=1$ & $r_0^*(g_{1, 0.7})=1.68326$ & $r_{1/2}^*(g_{1, 0.7})=1.24519$
    & $r_{0.95}^*(g_{1, 0.7})=0.410407$ \\ \hline
    $a=2$ & $r_0^*(g_{2, 0.7})=1.44678$ & $r_{1/2}^*(g_{2, 0.7})=1.1867$  & $r_{0.95}^*(g_{2, 0.7})=0.856647$\\  \hline
    $a=3$ & $r_0^*(g_{3, 0.7})=0.939782$ & $r_{1/2}^*(g_{3, 0.7})=0.763126$
    & $r_{0.95}^*(g_{3, 0.7})=0.549716$\\
    \hline
  \end{tabular}
  \vspace{.2in}
  \caption{The radius of starlikeness for $\mathtt{g}_{a, \nu}$ when $\nu=0.7$} \label{table-4}
\end{table}
\end{remark}
\begin{remark}
For $a=1$, Theorem \ref{thm-f-star} and Theorem \ref{thm-g-star} respectively reduces to Theorem 1 and Theorem 2 in \cite{BMD}.
\end{remark}
\section*{Funding}
The work of the first author was supported in parts by a FRGS research grant 203/PMATHS/6711568, while the second author was supported by a USM-RU grant 1001/PMATHS/8011015.

\end{document}